\documentclass[leqno]{amsart}
\usepackage{amsmath, amsthm, amssymb}
\usepackage{mathrsfs}
\usepackage{graphicx}

\newtheorem{thm}{Theorem}%[section]
\newtheorem{cor}{Corollary}
\newtheorem{lem}{Lemma}
\newtheorem{prop}{Proposition}

\theoremstyle{definition}

\makeatletter
\@addtoreset{equation}{section} 

\makeatother

\newcommand{\BC}{\mathbb{C}}

\newcommand{\BN}{\mathbb{N}}

\begin{document}

%
%
%  -------------------------------------------
%           << Title and Author >>
%  -------------------------------------------
%
%

\title[Uniqueness of the solution]
    {Uniqueness of the solution of some nonlinear singular 
       partial differential equations of the second order}

\author{Hidetoshi Tahara}
\address{Department of Information and 
        Communication Sciences, Sophia University, Tokyo 
        102-8554, Japan.}
\email{h-tahara@sophia.ac.jp}



%
%
%  -------------------------------------------
%           << Abstract >>
%  -------------------------------------------
%
%

\begin{abstract}
   In this paper, we consider a nonlinear singular second order partial 
differential equation of the form
\[
    \Bigl(t \frac{\partial}{\partial t} \Bigr)^2u
     = F \Bigl(t,x, \Bigl\{ 
         \Bigl(t \frac{\partial}{\partial t} \Bigr)^i
         \Bigl(\frac{\partial}{\partial x} \Bigr)^\alpha
           u \Bigr\}_{i+|\alpha| \leq 2,i<2} \Bigr)
\]
in the complex domain.  If $F(t,x,z)$ (with 
$z=\{z_{i,\alpha} \}_{i+|\alpha| \leq 2,i<2}$) is a holomorphic function
satisfying $F(0,x,0) \equiv 0$ and
$(\partial F/\partial z_{i,\alpha})(0,x,0)$ $\equiv 0$ (if $|\alpha|>0$), 
then this equation is called a nonlinear Fuchsian type partial 
differential equation in $t$. Under a very weak assumption, we show the 
uniqueness of the solution. The result is applied to the problem of 
analytic continuation of local holomorphic solutions of this equation.
\end{abstract}

\maketitle

%
%
%
%  -----------------------------------------------------
%              << \S 1. Introduction >>
%  -----------------------------------------------------
%

\section{Introduction}\label{section1}
%{\bf \S 1. Introduction.}
%

   To study the uniqueness of the solution is one of the most
fundamental problems in the theory of partial differential equations.
In this paper, we consider the case of nonlinear singular partial 
differential equations (\ref{1.1}) given below.
\par
   Let $m \in \BN^*$ ($=\{1,2,\ldots \}$). 
Let $(t,x)=(t,x_1,\ldots,x_n) \in \BC_t \times \BC_x^n$ be a complex 
variable and let 
$z=\{z_{i,\alpha}\}_{i+|\alpha| \leq m,i<m} \in \BC^N$ be also a  
complex variable, where we used the notations: $i \in \BN$ 
($=\{0,1,2,\ldots \}$), $\alpha=(\alpha_1,\ldots,\alpha_n) \in \BN^n$, 
$|\alpha|=\alpha_1+\cdots+\alpha_n$, and 
$N=\#\{(i,\alpha) \in \BN \times \BN^n \,;\, i+|\alpha| \leq m, i<m \}$.
We write $(\partial/\partial x)^{\alpha}
    = (\partial/\partial x_1)^{\alpha_1} \cdots 
            (\partial/\partial x_n)^{\alpha_n}$.
\par
   Let $F(t,x,z)$ be a function in a neighborhood $\Delta$ of the 
origin of $\BC_t \times \BC_x^n \times \BC_z^N$. 
Set $\Delta_0=\Delta \cap \{t=0, z=0 \}$.
In this paper, we consider a nonlinear singular partial differential 
equation
\begin{equation}\label{1.1}
    \Bigl(t \frac{\partial}{\partial t} \Bigr)^mu
     = F \Bigl(t,x,
      \Bigl\{ \Bigl(t \frac{\partial}{\partial t} \Bigr)^i
         \Bigl(\frac{\partial}{\partial x} \Bigr)^\alpha
           u \Bigr\}_{i+|\alpha| \leq m,i<m}  \Bigr)
\end{equation}
(in the germ sense at $(0,0) \in \BC_t \times \BC_x^n$) 
under the following assumptions:
\par
\medskip
   $\mbox{A}_1)$ \enskip $F(t,x,z)$ is a holomorphic function in $\Delta$.
\par
   $\mbox{A}_2)$ \enskip $F(0,x,0) \equiv 0$ in
       $\Delta_0$. 
\par
   $\mbox{A}_3)$ \enskip $(\partial F/\partial z_{i,\alpha})
           (0,x,0) \equiv 0$ in $\Delta_0$, if $|\alpha|>0$.
\par
\medskip
\noindent
Then, (\ref{1.1}) is called a nonlinear Fuchsian type partial 
differential equation with respect to $t$, and the roots 
$\lambda_1(x), \ldots, \lambda_m(x)$ of
\[
    \lambda^m- \sum_{i<m}(\partial F/\partial z_{i,0})
          (0,x,0) \lambda^i = 0
\]
are called the characteristic exponents of (\ref{1.1}). 
\par
   Equations of this type were first studied by G\'erard-Tahara 
\cite{gt, book}, and then by Tahara-Yamazawa \cite{yamazawa}. The most 
fundamental result on (\ref{1.1}) is:

\begin{thm}[G\'erard-Tahara \cite{gt}]\label{Theorem1}
    Suppose that $\lambda_i(0) \not\in \BN^*$ holds for $i=1,\ldots,m$. 
Then, the equation {\rm (\ref{1.1})} has a unique holomorphic 
solution $u_0(t,x)$ in a neighborhood of $(0,0) \in \BC_t \times \BC_x^n$ 
satisfying $u_0(0,x) \equiv 0$.
\end{thm}

   In this paper, we consider the uniqueness of the solution of 
(\ref{1.1}) under the assumption
\begin{equation}\label{1.2}
     {\rm Re}\lambda_i(0) < 0, \quad i=1,\ldots,m.
\end{equation}
We denote by ${\mathcal R}(\BC_t \setminus \{0\})$ the universal 
covering space of $\BC_t \setminus \{0\}$. For a nonempty open interval 
$I=(\theta_1,\theta_2)$ and $r>0$ we write 
$S_I=\{t \in {\mathcal R}(\BC \setminus \{0\}) \,;\,
  \theta_1<\arg t <\theta_2 \}$ and $S_I(r)=\{t \in S_I \,;\, 0<|t|<r \}$.
For $R>0$ we write $D_R=\{x \in \BC^n \,;\, |x|<R \}$, where
$|x|=\max_{1 \leq j \leq n}|x_j|$.
\par
   We know:
\par
   (1)(G\'erard-Tahara \cite{gt}). If (\ref{1.2}) is satisfied, and 
if $u(t,x)$ is a holomorphic solution of (\ref{1.1}) on 
$S_I(r_0) \times D_{R_0}$ (for some nonempty open interval $I$, $r_0>0$ 
and $R_0>0$) satisfying
\begin{equation}\label{1.3}
    \sup_{x \in D_{R_0}} |u(t,x)|= O(|t|^a) \quad 
    \mbox{(as $S_I \ni t \longrightarrow 0$)} 
\end{equation}
for some $a>0$, we have $u(t,x)=u_0(t,x)$ on $S_I(r) \times D_R$ for 
some $r>0$ and $R>0$. 
\par
   (2) In Tahara \cite{unique}, this condition (\ref{1.3}) was 
weakened to the following one: there is an $\epsilon >0$ such that
\[
    \sup_{x \in D_{R_0}} |u(t,x)|= O \Bigl( \frac{1}{|\log t|^{\epsilon}}
            \Bigr)
     \quad \mbox{(as $S_I \ni t \longrightarrow 0$)}.
\]

\par
   (3) After that, the author has been trying to prove the following 
conjecture as a working hypothesis.

\bigskip
{\bf Conjecture.}
%\begin{conj}\label{Conjecture}
    Suppose (\ref{1.2}). If $u(t,x)$ is a holomorphic solution 
of (\ref{1.1}) on $S_I(r_0) \times D_{R_0}$ (for some nonempty open 
interval $I$, $r_0>0$ and $R_0>0$) satisfying
\[
    \sup_{x \in D_{R_0}} |u(t,x)|= o(1) \quad 
    \mbox{(as $S_I \ni t \longrightarrow 0$)},
\]
we have $u(t,x)=u_0(t,x)$ on $S_I(r) \times D_R$ for some $r>0$ and 
$R>0$. 
%\end{conj}
%
%

\medskip
    The purpose of this paper is to show that the above conjecture
is true in the case $m=2$. The case $m=1$ has already been proved in 
Tahara \cite{briot} by a method similar to the Cauchy's characteristics 
method. In this paper, we will modify its argument so that it works 
also in the case $m=2$. Since the argument here works only in the case 
$m=2$, the above conjecture is still open in the case $m \geq 3$.

%
%
%  -----------------------------------------------------
%           << \S 2. Main result >>
%  -----------------------------------------------------
%

\section{Main result}\label{section2}
%\begin{center}
%{\bf \S 2. Main result}
%\end{center}
%

   From now, we consider the case $m=2$. Our equation is
\begin{equation}\label{2.1}
    \Bigl(t \frac{\partial}{\partial t} \Bigr)^2 u
     = F \Bigl(t,x,
      \Bigl\{ \Bigl(t \frac{\partial}{\partial t} \Bigr)^i
         \Bigl(\frac{\partial}{\partial x} \Bigr)^\alpha
           u \Bigr\}_{i+|\alpha| \leq 2, i<2} \Bigr).
\end{equation}
Let $\lambda_1(x), \lambda_2(x)$ be the characteristic exponents of 
(\ref{2.1}). We suppose the conditions A${}_1$), A${}_2$), A${}_3$) 
and 
\begin{equation}\label{2.2}
     {\rm Re}\lambda_i(0) < 0, \quad i=1,2. 
\end{equation}
The following result is the main theorem of this paper.

\begin{thm}\label{Theorem2}
    Let $u(t,x)$ be a holomorphic solution of {\rm (\ref{2.1})} 
on $S_I(r_0) \times D_{R_0}$ for some nonempty open interval $I$, 
$r_0>0$ and $R_0>0$. If $u(t,x)$ satisfies
\begin{equation}\label{2.3}
     \varlimsup_{R \to +0}\Bigl[\lim_{r \to +0} \, 
         \Bigl(\frac{1}{R^4}\sup_{S_I(r) \times D_R}|u(t,x)|
     \Bigr) \Bigr] \, = \, 0,
\end{equation}
we have $u(t,x)=u_0(t,x)$ on $S_I(\sigma) \times D_{\delta}$ for 
some $\sigma>0$ and $\delta>0$, where $u_0(t,x)$ is the unique 
holomorphic solution in Theorem {\rm \ref{Theorem1}}.
\end{thm}

   If 
\begin{equation}\label{2.4}
    \sup_{x \in D_R} |u(t,x)|= o(1), \quad 
    \mbox{(as $S_I \ni t \longrightarrow 0$)}
\end{equation}
holds for some $R>0$ we have (\ref{2.3}), and so we have

\begin{cor}\label{Corollary1}
    If $u(t,x)$ satisfies {\rm (\ref{2.4})}, we have 
$u(t,x)=u_0(t,x)$ on $S_I(\sigma) \times D_{\delta}$ for some 
$\sigma>0$ and $\delta>0$.
\end{cor}

   This proves the conjecture posed in \S 1 in the case $m=2$.

\medskip
{\bf Remark.}
%\begin{rem}\label{Remark}
    (1) If ${\rm Re}\lambda_1(0)>0$ or 
${\rm Re}\lambda_2(0)>0$ holds, we have many solutions satisfying 
(\ref{2.4}) and so the uniqueness of the solution is not valid. 
See \cite {gt, yamazawa}.
\par
   (2) In the case ${\rm Re}\lambda_1(0)=0$ or ${\rm Re}\lambda_2(0)=0$, 
we have the following counter example: the equation
\[
   \Bigl(t \, \frac{\partial }{\partial t} \Bigr)^2u
   = - \Bigl(t \, \frac{\partial }{\partial t} \Bigr)u
    + \Bigl(\frac{\partial u}{\partial x} \Bigr)^2
     + 8 u \Bigl(\frac{\partial^2 u}{\partial x^2} \Bigr)^2
\]
has a trivial solution $u(t,x) \equiv 0$ and a nontrivial solution
\[
    u(t,x) = \dfrac{-x^2}{4 \log t}
\]
which satisfies (\ref{2.4}). In this case, the characteristic exponents 
are $0$ and $-1$.
\par
   (3) The following example shows that the condition (\ref{2.3}) will 
be reasonable: the equation
\[
     \Bigl(t \frac{\partial}{\partial t} \Bigr)^2u
     = - 3 \Bigl(t \, \frac{\partial }{\partial t} \Bigr)u
      - 2u+ \Bigl(\frac{\partial^2 u}{\partial x^2} \Bigr)^2
\]
has a trivial solution $u(t,x) \equiv 0$ and a nontrivial solution 
$u=x^4/72$.  We note that for $u=x^4/72$ we have
\[
     \varlimsup_{R \to +0}\Bigl[\lim_{r \to +0} \, 
         \Bigl(\frac{1}{R^4}\sup_{S_I(r) \times D_R}|u(t,x)|
     \Bigr) \Bigr] \, = \, \frac{1}{72}.
\]
In this case, the characteristic exponents are $-1$ and $-2$.
%\end{rem}

%
%
%  -----------------------------------------------------
%       << \S 3. Some preparatory discussion >>
%  -----------------------------------------------------
%

\section{Some preparatory discussion}\label{section3}
%\begin{center}
%{\bf \S 3. Some preparatory discussion}
%\end{center}
%

   Before the proof of Theorem \ref{Theorem2}, let us present some 
preparatory discussion.  Let $u(t,x)$ be a holomorphic solution of 
(\ref{2.1}) on $S_I(\sigma_0^*) \times D_{R_0^*}$ (for some 
nonempty open interval $I$, $\sigma_0^*>0$ and $R_0^*>0$) satisfying 
the condition (\ref{2.3}). For simplicity, we set 
\begin{align*}
    &\Lambda=\{(i,\alpha) \in \BN \times \BN^n \,;\, 
             i+|\alpha| \leq 2, i<2 \}, \\
    &N=\# \Lambda \enskip \mbox{(the cardinal of $\Lambda$)}, \\
    &z=\{z_{i,\alpha} \}_{(i,\alpha) \in \Lambda} \quad
      \mbox{and} \quad z'=\{z_{0,\alpha} \}_{|\alpha|=2}.
\end{align*}
For $\nu=\{\nu_{i,\alpha} \}_{(i,\alpha) \in \Lambda} \in \BN^N$ we 
write
\[
    |\nu|= \sum_{(i,\alpha) \in \Lambda} \nu_{i,\alpha}, 
     \quad z^{\nu}= \prod_{(i,\alpha) \in \Lambda}
           (z_{i,\alpha})^{\nu_{i,\alpha}}.
\]
Suppose the conditions A${}_1$), A${}_2$), A${}_3$): then $F(t,x,z)$ is 
a holomorphic function in a neighborhood of the origin of 
$\BC_t \times \BC_x^n \times \BC_z^N$ having the expansion of the form 
\[
    F(t,x,z)= a(x)t + \beta_0^*(x)z_{0,0}+\beta_1^*(x)z_{1,0}
            +R(t,x,z)
\]
with
\[
      R(t,x,z)= \sum_{i+|\nu| \geq 2}\gamma_{i,\nu}^*(x)t^i z^{\nu}.
\]

   Step 1. We set $w(t,x)=u(t,x)-u_0(t,x)$. Then, we have
\begin{equation}\label{3.1}
     \varlimsup_{R \to +0}\Bigl[\lim_{r \to +0} \, 
         \Bigl(\frac{1}{R^4}\sup_{S_I(r) \times D_R}|w(t,x)|
     \Bigr) \Bigr] \, = \, 0,
\end{equation}
and $w(t,x)$ is a holomorphic solution of the equation 
\begin{equation}\label{3.2}
    \Bigl(t \frac{\partial}{\partial t} \Bigr)^2 w
     = H \Bigl(t,x, \Bigl\{ 
        \Bigl(t \frac{\partial}{\partial t} \Bigr)^i
         \Bigl(\frac{\partial}{\partial x} \Bigr)^{\alpha}
         w \Bigr\}_{(i,\alpha) \in \Lambda} \Bigr)
\end{equation}
on $S_I(\sigma_0^*) \times D_{R_0^*}$, where 
\begin{align*}
    H(t,x,z)= &F \bigl(t,x, z+
       \{ u_{0,i,\alpha}(t,x) \}_{(i,\alpha) \in \Lambda}
             \bigr) \\
        &\qquad \qquad \qquad   - F \bigl(t,x,
       \{ u_{0,i,\alpha}(t,x) \}_{(i,\alpha) \in \Lambda}
             \bigr)
\end{align*}
under the notation $u_{0,i,\alpha}(t,x)=(t \partial/\partial t)^i
(\partial/\partial x)^{\alpha}u_0(t,x)$ ($(i,\alpha) \in \Lambda$).
It is easy to see that $H(t,x,z)$ can be expressed in the form
\begin{align*}
    H(t,x,z)= &\beta_0^*(x)z_{0,0}+\beta_1^*(x)z_{1,0} 
          + t \sum_{(i,\alpha) \in \Lambda} 
             a_{i,\alpha}^*(t,x,z) z_{i,\alpha} \\
       &+ \sum_{(i,\alpha) \in \Lambda, |\alpha| \leq 1} 
             b_{i,\alpha}^*(x,z) z_{i,\alpha}
       + \sum_{|\alpha|=|\beta|=2}
          c_{\alpha,\beta}^*(x,z')z_{0,\alpha}z_{0,\beta}
\end{align*}
for some holomorphic functions $\beta_i^*(x)$ ($i=0,1$), 
$a_{i,\alpha}^*(t,x,z)$ ($(i,\alpha) \in \Lambda$),
$b_{i,\alpha}^*(x,z)$ ($(i,\alpha) \in \Lambda, |\alpha| \leq 1$),
and $c_{\alpha,\beta}^*(x,z')$ ($|\alpha|=|\beta|=2$) in a common 
neighborhood of $(0,0,0) \in \BC_t \times \BC_x^n \times \BC_z^N$ 
satisfying $b_{i,\alpha}^*(x,0) \equiv 0$ 
($(i,\alpha) \in \Lambda, |\alpha| \leq 1$). In addition, we see that 
the roots of $\lambda^2-\beta_1^*(x)\lambda-\beta_0^*(x)=0$ are just 
the same as $\lambda_1(x),\lambda_2(x)$.
\par
   From now, we consider the equation (\ref{3.2}). Our purpose is to 
show that $w(t,x) \equiv 0$ holds on $S_I(\sigma) \times D_\delta$ 
for some $\sigma>0$ and $\delta>0$.  

\par
\medskip
   Step 2. We set $\Theta_0=1$ and 
\begin{align*}
    &\Theta_1 = \Bigl( t \frac{\partial}{\partial t} 
              - \lambda_1(0) \Bigr), \\
    &\Theta_2 = 
    \Bigl( t \frac{\partial}{\partial t}- \lambda_2(0) \Bigr)
    \Bigl( t \frac{\partial}{\partial t}- \lambda_1(0) \Bigr).
\end{align*}
We write $Dw=\{D_{i,\alpha}w \}_{(i,\alpha) \in \Lambda}$ and 
$D'w=\{D_{0,\alpha}w \}_{|\alpha|=2}$ with
\[
    D_{i,\alpha}w= \Theta_i
         \Bigl(\frac{\partial}{\partial x} \Bigr)^\alpha w,
      \quad (i,\alpha) \in \Lambda.
\]
Then, (\ref{3.2}) is expressed in the form
\begin{align}
    \Theta_2 w &= \beta_0(x)w + \beta_1(x) \Theta_1w
      + t \sum_{(i,\alpha) \in \Lambda}
       a_{i,\alpha}(t,x,Dw) \Theta_i
       \Bigl(\frac{\partial}{\partial x} \Bigr)^\alpha w
                  \label{3.3} \\
      &+ \sum_{(i,\alpha) \in \Lambda, |\alpha| \leq 1}
       b_{i,\alpha}(x,Dw) \Theta_i
       \Bigl(\frac{\partial}{\partial x} \Bigr)^\alpha w \notag \\
      &+ \sum_{|\alpha|=|\beta|=2}
          c_{\alpha,\beta}(x,D'w) 
            \Bigl[\Bigl(\frac{\partial}{\partial x}
           \Bigr)^\alpha w \Bigr] \times 
            \Bigl[\Bigl(\frac{\partial}{\partial x} 
                    \Bigr)^\beta w \Bigr] \notag
\end{align}
for some holomorphic functions $\beta_i(x)$ ($i=0,1$), 
$a_{i,\alpha}(t,x,z)$ ($(i,\alpha) \in \Lambda$),
$b_{i,\alpha}(x,z)$ ($(i,\alpha) \in \Lambda, |\alpha| \leq 1$) and 
$c_{\alpha,\beta}(x,z')$ ($|\alpha|=|\beta|=2$) in a common 
neighborhood of $(0,0,0) \in \BC_t \times \BC_x^n \times \BC_z^N$ 
satisfying $\beta_i(0)=0$ ($i=0,1$) and $b_{i,\alpha}(x,0) \equiv 0$ 
($(i,\alpha) \in \Lambda, |\alpha| \leq 1$).

\par
\medskip
   Step 3.  By a rotation, we may suppose that $0 \in I$ holds: 
then we have $(0,\sigma_0^*) \subset S_I(\sigma_0^*)$. From now, we 
consider the equation (\ref{3.3}) only on 
$(0,\sigma_0^*) \times D_{R_0^*}$. By (\ref{3.1}) and Nagumo type 
lemma in a sector (for example, see Lemma 4.2 in \cite{bacani}) we 
have the condition
\begin{equation}\label{3.4}
     \varlimsup_{R \to +0}\Bigl[\lim_{\sigma \to +0} \, 
         \Bigl(\frac{1}{R^4}\sup_{(0,\sigma] \times D_R}
     \Bigl|\Bigl(t \frac{\partial }{\partial t} \Bigr)^iw(t,x) \Bigr|
     \Bigr) \Bigr] \, = \, 0, \quad i=0,1,2.
\end{equation}
\par
    For a formal power series 
$f(t,x) = \sum_{|\alpha| \geq 0} f_{\alpha}(t)x^\alpha$ with
coefficients in $C^0((0,T))$ we write
\[
    \| f(t) \|_{\rho} = \sum_{|\alpha| \geq 0} 
    |f_{\alpha}(t)| \frac{\alpha!}{|\alpha|!}\rho^{|\alpha|}
\]
(which is a formal power series in $\rho$ with coefficients
in $C^0((0,T))$). When $\| f(t) \|_{\rho}$ converges, we regard it 
as a function in $(t,\rho)$.  We can easily see:
\begin{align*}
   &\| f(t)g(t) \|_{\rho} \ll \| f(t) \|_{\rho}\| g(t) \|_{\rho}, \\
   &\| (\partial/\partial x_j)f(t) \|_{\rho}
     \ll (\partial/\partial \rho)\|f(t)\|_{\rho}, \quad j=1,\ldots,n.
\end{align*}
where $\sum_{i \geq 0}a_i \rho^i \ll \sum_{i \geq 0}b_i \rho^i$ means 
that $|a_i| \leq b_i$ holds for all $i \geq 0$.
For a holomorphic function 
$f(t,x,z)=\sum_{i+|\nu| \geq 0}f_{i,\nu}(x)t^iz^{\nu}$ we write
\[
     \|f\|_{\rho}(t,z)
     =\sum_{i+|\nu| \geq 0}\|f_{i,\nu}\|_{\rho}t^iz^{\nu}.
\]
\par
   By (\ref{3.4}) we have

\begin{lem}\label{Lemma1}
    Under the above situation, for $i=0,1,2$ we have
\begin{equation}\label{3.5}
     \lim_{\sigma \to +0} \, 
           \Bigl(\sup_{(0,\sigma] \times [0,R]} 
       \Bigl\| \Bigl(t \frac{\partial}{\partial t} \Bigr)^iw(t)
        \Bigr\|_{\rho} \Bigr) = o(R^4) \quad 
      \mbox{{\rm (}as $R \longrightarrow +0${\rm )}}.
\end{equation}
\end{lem}

\par
\medskip
   Step 4. Set $J=\{(i,j) \in \BN^2 \,;\, i+j \leq 2, i<2 \}$: 
actually, we have
\[
    J=\{(0,0), (1,0), (0,1), (1,1), (0,2)\}.
\]
For $(i,j) \in J$ we set
\begin{align*}
    &\phi_{0,0}(t,\rho) 
        = \int_{0}^{t} \, \Bigl( \dfrac{\tau}{t} \Bigr)
            ^{-{\rm Re}\lambda_1(0)} \| \Theta_1w(\tau) 
             \|_{\rho} \dfrac{d\tau}{\tau}, \\
    &\phi_{1,0}(t,\rho) 
        = \int_{0}^{t} \, \Bigl( \dfrac{\tau}{t} \Bigr)
            ^{-{\rm Re}\lambda_2(0)} \| \Theta_2w(\tau) 
             \|_{\rho} \dfrac{d\tau}{\tau}, \\
    &\phi_{0,1}(t,\rho)
         = \frac{\partial}{\partial \rho}\phi_{0,0}(t,\rho), \\
    &\phi_{1,1}(t,\rho)
               = \frac{\partial}{\partial \rho}\phi_{1,0}(t,\rho),\\
    &\phi_{0,2}(t,\rho)
            = \frac{\partial}{\partial \rho}\phi_{0,1}(t,\rho)
       = \Bigl(\frac{\partial}{\partial \rho} \Bigr)^2 
              \phi_{0,0}(t,\rho).
\end{align*}
By making $\sigma_0^*>0$ and $R_0^*>0$ smaller if necessary, we may 
suppose that $\phi_{i,j}(t,\rho)$ ($(i,j) \in J$) are convergent on 
$(0,\sigma_0^*] \times [0,R_0^*]$.  By (\ref{3.5}) we have

\begin{lem}\label{Lemma2}
    {\rm (1)} For any $(i,j) \in J$ we have
\[
     \lim_{\sigma \to +0} \, 
     \Bigl(\sup_{(0,\sigma] \times [0,R]} \phi_{i,j}(t,\rho)
      \Bigr) = o(R^{4-j}) \quad 
      \mbox{{\rm (}as $R \longrightarrow +0${\rm )}}.
\]
\par
   {\rm (2)} For any $\epsilon >0$ we can find $\sigma>0$ and $R>0$ 
such that $|\phi_{i,j}(t,\rho)| \leq \epsilon$ holds on 
$(0,\sigma] \times [0,R]$ for any $(i,j) \in J$.
\end{lem}

\par
\medskip
    Step 5. By (\ref{2.2}) we can take an $h>0$ such that
\[
     {\rm Re}\lambda_i(0) < -2h, \quad i=1,2. 
\]

\begin{lem}\label{Lemma3}
   {\rm (1)} For any $(i,\alpha) \in \Lambda$ we have
$$
    \|D_{i,\alpha}w(t) \|_{\rho}
    = \Bigl\| \Theta_i \Bigl(\frac{\partial}{\partial x}
            \Bigr)^{\alpha} w(t) \Bigr\|_{\rho}
    \ll \phi_{i,|\alpha|}(t,\rho) \quad \mbox{on $(0,\sigma_0^*]$}.
$$
\par
   {\rm (2)} We set 
$\Phi=\Phi(t,\rho)
    =(\phi_{i,|\alpha|}(t,\rho) \,;\, (i,\alpha) \in \Lambda)$
and $\Phi'=\Phi'(t,\rho)
    =(\phi_{0,|\alpha|}(t,\rho) \,;\, |\alpha|=2)$: 
we have
\begin{align*}
    &\Bigl(t \frac{\partial}{\partial t} + 2h \Bigr) 
      \phi_{0,0}(t,\rho) \ll \phi_{1,0}(t,\rho), \\
    &\Bigl(t \frac{\partial}{\partial t} + 2h \Bigr) 
      \phi_{1,0}(t,\rho) \\
    &\qquad \ll \|\beta_0\|_{\rho} \phi_{0,0}(t,\rho)
               + \|\beta_1\|_{\rho} \phi_{1,0}(t,\rho) 
       + t \sum_{(i,\alpha) \in \Lambda} 
       \|a_{i,\alpha}\|_{\rho}(t,\Phi) \phi_{i,|\alpha|}(t,\rho) \\
    &\qquad + \!\!\! \sum_{(i,\alpha) \in \Lambda, |\alpha| \leq 1}
       \!\!\! \|b_{i,\alpha}\|_{\rho}(\Phi) \phi_{i,|\alpha|}(t,\rho)
     +  \sum_{|\alpha|=|\beta|=2}
    \|c_{\alpha, \beta} \|_{\rho}(\Phi')(\phi_{0,2}(t,\rho))^2
\end{align*}
on $(0,\sigma_0^*]$.
\par
   {\rm (3)} By taking $R_0^*>0$ a smaller one if necessary, we may 
assume that $\|\beta_i\|_{\rho}$ {\rm (}$i=0,1${\rm )}, 
$\|a_{i,\alpha}\|_{\rho}(t,z)$
{\rm (}$(i,\alpha) \in \Lambda${\rm )}, $\|b_{i,\alpha}\|_{\rho}(z)$ 
{\rm (}$(i,\alpha) \in \Lambda, |\alpha| \leq 1${\rm )} and 
$\|c_{\alpha,\beta} \|_{\rho}(z')$ 
are convergent on $(0,\sigma_0^*] \times [0,R_0^*] \times D_L$
{\rm (}where $D_L= \{z \in \BC^N \,;\, |z_{i,\alpha}|<L \,\, 
              ((i,\alpha) \in \Lambda) \}${\rm )} for some $L>0$. 
In addition, we have 
\begin{align*}
    &\|\beta_i\|_{\rho} \leq H_i \rho \quad \mbox{on $[0,R_0^*]$},
            \quad i=0,1, \\
    &\|b_{i,\alpha}\|_{\rho}(z) \leq B_{i,\alpha}|z|
           \quad \mbox{on $[0,R_0^*] \times D_L$}, \quad 
             (i,\alpha) \in \Lambda, \, |\alpha| \leq 1
\end{align*}
for some constants $H_i>0$ {\rm (}$i=0,1${\rm )} and $B_{i,\alpha}>0$ 
{\rm (}$(i,\alpha) \in \Lambda, |\alpha| \leq 1${\rm )}, 
where $|z|=\sum_{(i,\alpha) \in \Lambda}|z_{i,\alpha}|$.
\par
   {\rm (4)} For any $(i,\alpha) \in \Lambda$ with $|\alpha| \leq 1$
we have
\[
     \lim_{\sigma \to +0} \, 
           \Bigl(\sup_{(0,\sigma] \times [0,R]} 
       \frac{\partial \|b_{i,\alpha}\|_{\rho}(\Phi)}
                                 {\partial \rho} \Bigr) = o(R)
       \quad \mbox{{\rm (}as $R \longrightarrow +0${\rm )}}.
\]
\end{lem}

\begin{proof}
    Let us show (1). Since 
\[
    \Bigl(t \frac{\partial}{\partial t} 
          - \lambda_{i+1}(0) \Bigr) \Theta_iw = (\Theta_{i+1}w) (t,x),
    \quad i=0,1
\]
hold, by integrating this we have
\[
     \Theta_iw(t,x)
     = \int_0^t \, \Bigl( \dfrac{\tau}{t} \Bigr)
            ^{- \lambda_{i+1}(0)} (\Theta_{i+1} w)(\tau,x) 
            \dfrac{d\tau}{\tau}
\]
and so by taking the norm 
\[
      \|\Theta_i w(t) \|_{\rho}
     \ll \int_0^t 
         \Bigl( \dfrac{\tau}{t} \Bigr)^{- {\rm Re}\lambda_{i+1}(0)}
         \|\Theta_{i+1} w(\tau)\| \dfrac{d\tau}{\tau}
         = \phi_{i,0}(t,\rho), \quad i=0,1.
\]
Hence, for any $(i,\alpha) \in \Lambda$ we have
\[
    \Bigl\|\Theta_i \Bigl(\dfrac{\partial}{\partial x}
           \Bigr)^{\alpha}w(t) \Bigr\|_{\rho}
    \ll \Bigl(\frac{\partial}{\partial \rho} \Bigr)^{|\alpha|}
             \|\Theta_i w(t) \|_{\rho}
    \ll \Bigl(\frac{\partial}{\partial \rho} \Bigr)^{|\alpha|}
         \phi_{i,0}(t,\rho)
      = \phi_{i,|\alpha|}(t,\rho).
\]
\par
   Let us show (2). We have
\begin{align*}
   \Bigl(t \frac{\partial}{\partial t} + 2h \Bigr) 
          \phi_{0,0}(t,\rho) 
   &\ll \Bigl(t \frac{\partial}{\partial t} + {\rm Re}\lambda_1(0)
      \Bigr) \phi_{0,0}(t,\rho) \\
   &= \| \Theta_1w(\tau) \|_{\rho} \ll \phi_{1,0}(t,\rho) 
                   \quad \mbox{on $(0,\sigma_0^*]$}.
\end{align*}
This proves the first inequality of (2). Since 
\[
   \Bigl(t \frac{\partial}{\partial t} + 2h \Bigr) 
          \phi_{1,0}(t,\rho) 
   \ll \Bigl(t \frac{\partial}{\partial t} + {\rm Re}\lambda_2(0)
      \Bigr) \phi_{1,0}(t,\rho) 
   = \| \Theta_2w(\tau) \|_{\rho},
\]
by applying (\ref{3.3}) and by using (1) we can easily see the second 
inequality of (2). The condition (3) is clear. Since 
$\|b_{j,\alpha}\|_{\rho}(\Phi) \leq B_{j,\alpha}|\Phi|$ 
($(j,\alpha) \in \Lambda$) hold, by (1) of Lemma 2 we have
\[
     \lim_{\sigma \to +0} \, 
           \Bigl(\sup_{(0,\sigma] \times [0,R]} 
       \|b_{i,\alpha}\|_{\rho}(\Phi) \Bigr) = o(R^2)
       \quad \mbox{(as $R \longrightarrow +0$)}.
\]
This leads us to the condition (4).
\end{proof}

   Step 6.  Let $\varepsilon_{0,0}>0$, $\varepsilon_{1,0}=1$, 
$\varepsilon_{0,1}>0$, $\varepsilon_{1,1}>0$ and $0<\kappa<1$. We set
\begin{align*}
    q(t,\rho)= &\varepsilon_{0,0} \phi_{0,0}(t,\rho)
      + \phi_{1,0}(t,\rho)
    + t^\kappa \phi_{0,2}(t,\rho) \\
          & + \varepsilon_{0,1} \phi_{0,1}(t,\rho)
                +\varepsilon_{1,1} \phi_{1,1}(t,\rho) 
          + ( \phi_{0,2}(t,\rho))^{3/2}.
\end{align*}

\begin{lem}\label{Lemma4}
    We have the following inequality:
\begin{equation}\label{3.6}
   \Bigl(t \frac{\partial}{\partial t} + 2h \Bigr) q
    \leq A(t,\rho)\, q + 
         B(t,\rho)\frac{\partial}{\partial \rho} q
\end{equation}
on $(0,\sigma_0^*] \times [0,R_0^*]$, where 
\begin{align*}
    A(t,\rho) &= \varepsilon_{0,0}
           + \Bigl( \frac{1}{\varepsilon_{0,0}}\|\beta_0\|_{\rho}
                    + \|\beta_1\|_{\rho} \Bigr) \\
    &+ \sum_{(i,\alpha) \in \Lambda, |\alpha| \leq 1}
           \frac{1}{\varepsilon_{i,|\alpha|}} \times 
           t\, \|a_{i,\alpha}\|_{\rho}(t,\Phi)
      + \sum_{|\alpha|=2}t^{1-\kappa}\|a_{0,\alpha}\|_{\rho}(t,\Phi) \\
    &+ \sum_{(i,\alpha) \in \Lambda, |\alpha| \leq 1}
       \frac{1}{\varepsilon_{i,|\alpha|}} \times 
          \|b_{i,\alpha}\|_{\rho}(\Phi)
         + \sum_{|\alpha|=|\beta|=2}
         \|c_{\alpha,\beta} \|_{\rho}(\Phi')(\phi_{0,2})^{1/2}  \\
    &+ \kappa 
               + \frac{\varepsilon_{0,1}}{\varepsilon_{1,1}}
    %%%%%%%%%%%%%%%%%%%%%%%%%%%%%%%%%%%%%
      +  \varepsilon_{1,1} \Bigl(\frac{1}{\varepsilon_{0,0}}
       \frac{\partial \|\beta_0\|_{\rho}}{\partial \rho} 
    + \frac{\partial \|\beta_1\|_{\rho}}{\partial \rho} \Bigr) \\
    &+ \varepsilon_{1,1} \Bigl(
        \frac{\|\beta_0\|_{\rho}}{\varepsilon_{0,1}}  
        + \frac{\|\beta_1\|_{\rho}}{\varepsilon_{1,1}} \Bigr)
    +\sum_{(i,\alpha) \in \Lambda, |\alpha| \leq 1}
        \frac{\varepsilon_{1,1}}{\varepsilon_{i,|\alpha|}} \times
       t\, \frac{\partial \|a_{i,\alpha}\|_{\rho}(t,\Phi)}
                                 {\partial \rho} \\
    &+ \varepsilon_{1,1}
           \sum_{|\alpha|=2}t^{1-\kappa}
             \frac{\partial \|a_{0,\alpha}\|_{\rho}(t,\Phi)}
                          {\partial \rho}
        +\sum_{(i,\alpha) \in \Lambda, |\alpha| \leq 1}
        \frac{\varepsilon_{1,1}}{\varepsilon_{i,|\alpha|}} \times
            \frac{\partial \|b_{i,\alpha}\|_{\rho}(\Phi)}
                                 {\partial \rho} \\
    &+ \varepsilon_{1,1} \sum_{|\alpha|=|\beta|=2}
     \frac{\partial \|c_{\alpha,\beta}\|_{\rho}(\Phi')}{\partial \rho}
                (\phi_{0,2})^{1/2}, \\
%\end{align*}
%\begin{align*}
%%%%%%%%%%%%%%%%%%%%%%%%%%%%%%%%%%%%%%%%%%%%%%%%%%%%%%%%%%
    B(t,\rho) &= \frac{t^{\kappa}}{\varepsilon_{1,1}} 
      + \sum_{(i,\alpha) \in \Lambda, |\alpha| \leq 1}
        \frac{\varepsilon_{1,1}}{\varepsilon_{i,|\alpha|}} \times
             t\, \|a_{i,\alpha}\|_{\rho}(t,\Phi) \\
        &+ \varepsilon_{1,1} \sum_{|\alpha|=2}
            t^{1-\kappa}\|a_{0,\alpha}\|_{\rho}(t,\Phi) 
    + \sum_{(i,\alpha) \in \Lambda, |\alpha| \leq 1}
        \frac{\varepsilon_{1,1}}{\varepsilon_{i,|\alpha|}} \times
             \|b_{i,\alpha}\|_{\rho}(\Phi) \\
    &+ \frac{4 \varepsilon_{1,1}}{3}
       \sum_{|\alpha|=|\beta|=2}\|c_{\alpha,\beta} \|_{\rho}
            (\Phi') (\phi_{0,2})^{1/2}
         + \frac{3}{2\varepsilon_{1,1}}(\phi_{0,2})^{1/2}.
\end{align*}
\end{lem}

\begin{proof}
    By the definition of $q(t,\rho)$ we have 
\begin{equation}\label{3.7}
   \phi_{i,j} \leq \frac{1}{\varepsilon_{i,j}} q \quad
       (i,j=0,1), 
        \quad  \phi_{0,2} \leq q^{2/3}, 
        \quad t^{\kappa} \phi_{0,2} \leq q 
\end{equation}
on $(0,\sigma_0^*] \times [0,R_0^*]$. Since
\begin{align*}
    \frac{\partial}{\partial \rho}q
      = &\varepsilon_{0,0} 
          \frac{\partial}{\partial \rho}\phi_{0,0}
         + \frac{\partial}{\partial \rho} \phi_{1,0} 
      + t^\kappa \frac{\partial}{\partial \rho} \phi_{0,2} \\
       &+ \varepsilon_{0,1} 
             \frac{\partial}{\partial \rho} \phi_{0,1} 
       + \varepsilon_{1,1} 
             \frac{\partial}{\partial \rho} \phi_{1,1} 
        + \frac{3}{2}(\phi_{0,2})^{1/2}
            \frac{\partial}{\partial \rho} \phi_{0,2}
\end{align*}
we have also
\begin{equation}\label{3.8}
   \begin{split}
   &\frac{\partial}{\partial \rho}\phi_{i,j} 
        \leq \frac{1}{\varepsilon_{i,j}}
                     \frac{\partial}{\partial \rho}q
       \quad (i,j=0,1), \\
   &(\phi_{0,2})^{1/2} \frac{\partial}{\partial \rho} \phi_{0,2}
         \leq \frac{2}{3} \frac{\partial}{\partial \rho} q,
     \quad  t^\kappa \frac{\partial}{\partial \rho} \phi_{0,2}
               \leq \frac{\partial}{\partial \rho} q 
    \end{split}
\end{equation}
on $(0,\sigma_0^*] \times [0,R_0^*]$. By using these inequality, 
let us do a calculation.

\par
   1) About $\varepsilon_{0,0} \phi_{0,0}(t,\rho)$, by (2) 
of Lemma \ref{Lemma3} and (\ref{3.7}) we have
\[
    \Bigl(t \frac{\partial}{\partial t} + 2h \Bigr)
                  (\varepsilon_{0,0} \phi_{0,0})
     \ll \varepsilon_{0,0} \phi_{1,0} \leq \varepsilon_{0,0} q.
\]
\par
   2) About $\phi_{1,0}$, by (2) of Lemma \ref{Lemma3} and 
(\ref{3.7}) we have
\begin{align*}
    &\Bigl(t \frac{\partial}{\partial t} + 2h \Bigr) 
      \phi_{1,0} \\
    &\ll \|\beta_0\|_{\rho} \phi_{0,0}
               + \|\beta_1\|_{\rho} \phi_{1,0} 
       + t \sum_{(i,\alpha) \in \Lambda} 
       \|a_{i,\alpha}\|_{\rho}(t,\Phi) \phi_{i,|\alpha|} \\
    &\qquad + \!\!\! \sum_{(i,\alpha) \in \Lambda, |\alpha| \leq 1}
       \!\!\! \|b_{i,\alpha}\|_{\rho}(\Phi) \phi_{i,|\alpha|}
         + \sum_{|\alpha|=|\beta|=2}
    \|c_{\alpha, \beta} \|_{\rho}(\Phi')(\phi_{0,2})^2 \\
    &\leq \Bigl( \frac{1}{\varepsilon_{0,0}}\|\beta_0\|_{\rho}
                    + \|\beta_1\|_{\rho} \Bigr)q \\
    &\qquad + \sum_{(i,\alpha) \in \Lambda, |\alpha| \leq 1}
         \frac{1}{\varepsilon_{i,|\alpha|}} \,t \,
        \|a_{i,\alpha}\|_{\rho}(t,\Phi) \, q
         + t^{1-\kappa} \sum_{|\alpha|=2}
             \|a_{0,\alpha}\|_{\rho}(t,\Phi) q \\
    &\qquad + \sum_{(i,\alpha) \in \Lambda, |\alpha| \leq 1}
                 \frac{1}{\varepsilon_{i,|\alpha|}} 
             \|b_{i,\alpha}\|_{\rho}(\Phi) \, q
         + \sum_{|\alpha|=|\beta|=2}
           \|c_{\alpha,\beta} \|_{\rho}(\Phi')
             (\phi_{0,2})^{1/2} q.
\end{align*}
\par
   3) Let us consider $t^{\kappa} \phi_{0,2}(t,\rho)$. Since 
$(\partial/\partial \rho)\phi_{1,1} 
      \ll (1/\varepsilon_{1,1}) (\partial/\partial \rho)q$ 
holds, we have 
\[
    \Bigl(t \frac{\partial}{\partial t} + 2h \Bigr)\phi_{0,2}
    = \Bigl(\frac{\partial}{\partial \rho} \Bigr)^2
          \Bigl(t \frac{\partial}{\partial t} + 2h \Bigr)\phi_{0,0}
    \ll \Bigl(\frac{\partial}{\partial \rho} \Bigr)^2 \phi_{1,0}
    = \frac{\partial}{\partial \rho}\phi_{1,1}
      \leq \frac{1}{\varepsilon_{1,1}}\frac{\partial}{\partial \rho} q
\]
and so
\begin{align*}
   &\Bigl(t \frac{\partial}{\partial t} + 2h \Bigr) 
            t^{\kappa} \phi_{0,2}
    = t^{\kappa} \Bigl(t \frac{\partial}{\partial t} + 2h \Bigr) 
          \phi_{0,2} + \kappa t^{\kappa} \phi_{0,2} 
   \leq \frac{t^{\kappa}}{\varepsilon_{1,1}}
              \frac{\partial}{\partial \rho} q + \kappa q.
\end{align*}
\par
   4) About $\varepsilon_{0,1} \phi_{0,1}(t,\rho)$, we have
\begin{align*}
    \Bigl(t \frac{\partial}{\partial t} + 2h \Bigr)
           (\varepsilon_{0,1}\phi_{0,1})
    &= \varepsilon_{0,1}\frac{\partial}{\partial \rho}
       \Bigl(t \frac{\partial}{\partial t} + 2h \Bigr)\phi_{0,0} \\
    &\ll \varepsilon_{0,1} \frac{\partial}{\partial \rho}\phi_{1,0}
     = \varepsilon_{0,1}\phi_{1,1}
    \leq \frac{\varepsilon_{0,1}}{\varepsilon_{1,1}} q.
\end{align*}
\par
   5) About $\varepsilon_{1,1} \phi_{1,1}(t,\rho)$, by (2) of 
Lemma \ref{Lemma3} we have
\begin{align*}
    &\Bigl(t \frac{\partial}{\partial t} + 2h \Bigr)
           (\varepsilon_{1,1}\phi_{1,1})
    = \varepsilon_{1,1}\frac{\partial}{\partial \rho}
       \Bigl(t \frac{\partial}{\partial t} + 2h \Bigr)\phi_{1,0} \\
    &\ll \varepsilon_{1,1} \frac{\partial}{\partial \rho}
     \biggl[\|\beta_0\|_{\rho} \phi_{0,0}
               + \|\beta_1\|_{\rho} \phi_{1,0} 
       + t \sum_{(i,\alpha) \in \Lambda} 
       \|a_{i,\alpha}\|_{\rho}(t,\Phi) \phi_{i,|\alpha|} \\
    &+ \!\!\! \sum_{(i,\alpha) \in \Lambda, |\alpha| \leq 1}
       \!\!\! \|b_{i,\alpha}\|_{\rho}(\Phi) \phi_{i,|\alpha|}
         + \sum_{|\alpha|=|\beta|=2}\|c_{\alpha,\beta} \|_{\rho}
                (\Phi')(\phi_{0,2})^2 \biggr].
\end{align*}
By calculating the right side of the above formula and then by 
using (\ref{3.7}) and (\ref{3.8}) we have
\begin{align*}
    &\Bigl(t \frac{\partial}{\partial t} + 2h \Bigr)
           (\varepsilon_{1,1}\phi_{1,1}) \\
    &\leq 
       \varepsilon_{1,1}\Bigl(\frac{1}{\varepsilon_{0,0}}  
       \frac{\partial \|\beta_0\|_{\rho}}{\partial \rho}
       + \frac{\partial \|\beta_1\|_{\rho}}{\partial \rho} \Bigr) q
     + \varepsilon_{1,1} \Bigl(
      \frac{\|\beta_0\|_{\rho}}{\varepsilon_{0,1}}  
      + \frac{\|\beta_1\|_{\rho}}{\varepsilon_{1,1}} \Bigr)q \\
    &+\sum_{(i,\alpha) \in \Lambda, |\alpha| \leq 1}
        \frac{\varepsilon_{1,1}}{\varepsilon_{i,|\alpha|}} \times
       t \frac{\partial \|a_{i,\alpha}\|_{\rho}(t,\Phi)}
                                 {\partial \rho} \times q \\
    &+ \sum_{(i,\alpha) \in \Lambda, |\alpha| \leq 1}
        \frac{\varepsilon_{1,1}}{\varepsilon_{i,|\alpha|}} \times
             t \|a_{i,\alpha}\|_{\rho}(t,\Phi) 
              \frac{\partial}{\partial \rho} q \\
    &+ \varepsilon_{1,1}
          t^{1-\kappa} \sum_{|\alpha|=2} 
           \frac{\partial \|a_{0,\alpha}\|_{\rho}(t,\Phi)}
                          {\partial \rho} q 
         + \varepsilon_{1,1}t^{1-\kappa} \sum_{|\alpha|=2}
         \|a_{0,\alpha}\|_{\rho}(t,\Phi) 
              \frac{\partial}{\partial \rho} q  \\
    &+\sum_{(i,\alpha) \in \Lambda, |\alpha| \leq 1}
        \frac{\varepsilon_{1,1}}{\varepsilon_{i,|\alpha|}} \times
            \frac{\partial \|b_{i,\alpha}\|_{\rho}(\Phi)}
                                 {\partial \rho} q \, 
     + \sum_{(i,\alpha) \in \Lambda, |\alpha| \leq 1}
        \frac{\varepsilon_{1,1}}{\varepsilon_{i,|\alpha|}} \times
             \|b_{i,\alpha}\|_{\rho}(\Phi) 
              \frac{\partial}{\partial \rho} q \\
    &+ \varepsilon_{1,1} \sum_{|\alpha|=|\beta|=2}
    \frac{\partial \|c_{\alpha,\beta}\|_{\rho}(\Phi')}{\partial \rho}
                (\phi_{0,2})^{1/2} q \\
    &+ \frac{4 \varepsilon_{1,1}}{3} \sum_{|\alpha|=|\beta|=2}
          \|c_{\alpha,\beta} \|_{\rho}(\Phi') (\phi_{0,2})^{1/2}
          \frac{\partial}{\partial \rho} q
\end{align*}
\par
   6) Let us consider $(\phi_{0,2})^{3/2}$. Since
\[
    \Bigl(t \frac{\partial}{\partial t} + 2h \Bigr)\phi_{0,2}
    = \frac{\partial}{\partial \rho} 
          \Bigl(t \frac{\partial}{\partial t} + 2h \Bigr)\phi_{0,1}
    \ll \frac{\partial}{\partial \rho}\phi_{1,1}
    \leq \frac{1}{\varepsilon_{1,1}}
                    \frac{\partial}{\partial \rho} q
\]
we have
\begin{align*}
   &\Bigl(t \frac{\partial}{\partial t} + 2h \Bigr) 
      (\phi_{0,2})^{3/2} 
       = \frac{3}{2}(\phi_{0,2})^{1/2}
          t \frac{\partial}{\partial t} \phi_{0,2}+
               2h (\phi_{0,2})^{3/2} \\
   &\leq \frac{3}{2}(\phi_{0,1})^{1/2}
       \Bigl(t \frac{\partial}{\partial t} + 2h \Bigr)\phi_{0,2}
     \leq \frac{3}{2\varepsilon_{1,1}}(\phi_{0,2})^{1/2}
                \frac{\partial}{\partial \rho} q.
\end{align*}
\par
   7) Thus, by taking the summation from 1) to 6) we have the 
result (\ref{3.6}). 
\end{proof}

\par
   Step 7. Let us estimate $A(t,\rho)$ and $B(t,\rho)$.
We have

\begin{lem}\label{Lemma5}
    By taking $\varepsilon_{0,0}>0$, 
$\varepsilon_{0,1}>0$, $\varepsilon_{1,1}>0$, $0<\kappa<1/2$, 
$\sigma_0>0$ and $R_0>0$ suitably, we have
\begin{align}
    &A(t,\rho) \leq h, \label{3.9} \\
    &B(t,\rho) \leq 
       C_1t^{\kappa} +C_2q + C_3 q^{2/3}+C_4 q^{1/3}
             \label{3.10}
\end{align}
on $(0,\sigma_0] \times[0,R_0]$ for some $C_i>0$ 
{\rm (}$i=1,2,3,4${\rm )}.
\end{lem}

\begin{proof}
   1) First, we take $\varepsilon_{0,0}$ so that 
$0<\varepsilon_{0,0} \leq h/4$.
\par
   2) Since $\varepsilon_{0,0}$ is a fixed constant, by taking 
$\varepsilon_{1,1}>0$ sufficiently small we have the condition
\[
   \varepsilon_{1,1} \Bigl(\frac{1}{\varepsilon_{0,0}}  
       \frac{\partial \|\beta_0\|_{\rho}}{\partial \rho}
       + \frac{\partial \|\beta_1\|_{\rho}}{\partial \rho} \Bigr)
    \leq \frac{h}{4} \quad 
       \mbox{on $(0,\sigma_0^*] \times [0,R_0^*]$}.
\]
\par
   3) Thirdly, we take $\kappa>0$ and $\varepsilon_{0,1}>0$ 
so that $\kappa+\varepsilon_{0,1}/\varepsilon_{1,1} \leq h/4$. 
\par
   4) Now, $\varepsilon_{i,|\alpha|}>0$ 
($(i,\alpha) \in \Lambda, |\alpha| \leq 1$) are fixed. Since 
$\|\beta_i\|_{\rho} \leq H_i \rho$ ($i=0,1$) and 
$\|b_{i,\alpha}\|_{\rho}(\Phi) \leq B_{i,\alpha}|\Phi|$ 
($(i,\alpha) \in \Lambda, |\alpha| \leq 1$) are known 
(by (3) of Lemma \ref{Lemma3}), by the conditions (2) of Lemma \ref{Lemma2}
and (4) of Lemma \ref{Lemma3} we can take $\sigma_0>0$ and $R_0>0$ 
suitably sufficiently small so that the condition (\ref{3.9}) is valid 
on $(0,\sigma_0] \times [0,R_0]$.
\par
    5) Since $0<\kappa<1/2$ is supposed, we have $t=O(t^{\kappa})$ and 
$t^{1-\kappa}=O(t^{\kappa})$. Hence, by the condition 
$\|b_{i,\alpha}\|_{\rho}(\Phi) \leq B_{i,\alpha}|\Phi|$ we have
\[
    B(t,\rho) \leq 
       K_1t^{\kappa} +K_2 |\Phi| + K_3 (\phi_{0,2})^{1/2}
\]
for some $K_i>0$ ($i=1,2,3$). By applying (\ref{3.7}) to this estimate 
we have the condition (\ref{3.10}).
\end{proof}

%
%
%  -----------------------------------------------------
%       << \S 4. Proof of Theorem 2 >>
%  -----------------------------------------------------
%

\section{Proof of Theorem 2}\label{section4}
%\begin{center}
%{\bf \S 4. Proof of Theorem 2}
%\end{center}
%

   In order to prove Theorem \ref{Theorem2}, it is enough to show

\begin{prop}\label{Proposition1}
      Under the situation in \S 3, there are 
$\sigma>0$ and $\delta>0$ such that $q(t,\rho)=0$ holds on 
$(0,\sigma) \times [0,\delta)$.
\end{prop}

\begin{proof}
    Let us show this, step by step.

   Step 1. By Lemmas \ref{Lemma4} and \ref{Lemma5} we have
\begin{align}
    &\Bigl(t \frac{\partial}{\partial t} + h \Bigr) q
    \leq B(t,\rho)\frac{\partial}{\partial \rho} q,
             \label{4.1}\\
    &B(t,\rho) \leq C_1t^{\kappa}
       + C_2 q + C_3 q^{2/3}+ C_4 q^{1/3} \label{4.2}
\end{align}
on $(0,\sigma_0] \times [0,R_0]$. For $\sigma>0$ and $R>0$ we set
\[
      r= \sup_{(0,\sigma] \times [0,R]} q(t,\rho).
\]
By (1) of Lemma \ref{Lemma2} and the definition of $q(t,\rho)$ we have
\[
     \lim_{\sigma \to +0} \,r \, = \, o(R^3)
       \quad \mbox{{\rm (}as $R \longrightarrow +0${\rm )}}.
\]
This shows

\begin{lem}\label{Lemma6}
     By taking $\sigma>0$ and $R>0$ sufficiently small, 
we have the condition
\[
    \frac{C_1}{\kappa}\sigma^{\kappa}
     + \frac{C_2}{h}r + \frac{3C_3}{2h}r^{2/3} 
           + \frac{3C_4}{h}r^{1/3} < \frac{R}{2} \quad 
    \mbox{on $(0,\sigma] \times [0,R]$}.
\]
\end{lem}

    Step 2. Let $\sigma>0$ and $R>0$ be as in Lemma \ref{Lemma6}. 
Take any $t_0 \in (0,\sigma]$ and $\xi \in (0,R)$; for a while we fix 
them.
\par
    Let us consider the equation
\begin{equation}\label{4.3}
     \frac{d\rho}{dt} = - \frac{B(t,\rho)}{t}, 
       \quad \rho(t_0)=\xi 
\end{equation}
in the region $(0,t_0] \times (0,R)$. Since $B(t,\rho)/t$ 
is a continuous function on $(0,t_0] \times (0,R)$ we have a local 
solution (not necessarily unique). We take a maximally extended 
solution $\rho(t)$ and we denote by $(t_{\xi},t_0]$ the interval of 
existence of this maximally extended solution. Set
\[
    q^*(t)= q(t,\rho(t)), \quad t_{\xi}<t \leq t_0.
\]

\begin{lem}\label{Lemma7}
    Under the above situation, we have the following 
inequality for any $(t_1,\tau)$ satisfying $t_\xi<t_1<\tau \leq t_0$:
\begin{equation}\label{4.4}
     q^*(\tau) \leq \Bigl( \frac{t_1}{\tau} \Bigr)^h q^*(t_1).
\end{equation}
\end{lem}

\begin{proof}
    By (\ref{4.1}) we have
\begin{align*}
    \Bigl(t \frac{d}{dt}+h \Bigr) q^*(t) 
    &=\Bigl(t \frac{\partial}{\partial t}+h \Bigr)
                              q(t, \rho)\Bigr|_{\rho=\rho(t)}
        +\frac{\partial q}{\partial \rho}(t,\rho(t))
          \times t \frac{d\rho(t)}{dt} \\
    &\leq \Bigl[B(t,\rho)
              \frac{\partial q}{\partial \rho}
    + \frac{\partial q}{\partial \rho} 
        \bigl(-B(t,\rho) \bigr) \Bigr] \Bigr|_{\rho=\rho(t)}
    =0,
\end{align*}
that is, 
\[
    \Bigl(t \frac{d}{dt}+h \Bigr) q^*(t) \leq 0
     \quad \mbox{on $(t_{\xi},t_0]$}.
\]
Since this is equivalent to
\[
     \frac{d}{dt}(t^hq^*(t)) \leq 0
     \quad \mbox{on $(t_{\xi},t_0]$}, 
\]
by integrating this from $t_1$ to $\tau$ we have
$\tau^h q^*(\tau) \leq {t_1}^h q^*(t_1)$. This proves (\ref{4.4}).
\end{proof}

   Step 3.  By (\ref{4.2}) and Lemma \ref{Lemma7} we have

\begin{lem}\label{Lemma8}
    Under the above situation, we have the following 
inequalities for any $t_1$ satisfying $t_\xi<t_1< t_0$:
\begin{equation}\label{4.5}
     \xi \leq \rho(t_1) \leq \xi+ \frac{C_1}{\kappa}
          ({t_0}^{\kappa}-{t_1}^{\kappa})
       + \frac{C_2}{h} r + \frac{3C_3}{2h} r^{2/3}
       + \frac{3C_4}{h} r^{1/3}. 
\end{equation}
\end{lem}

\begin{proof}
   Let $t_\xi<t_1< t_0$. By (\ref{4.3}) we have
\[
      \rho(t_1)= \xi + \int_{t_1}^{t_0}
           B(\tau,\rho(\tau)) \frac{d\tau}{\tau}.
\]
Since $B(t,\rho) \geq 0$ we have $\rho(t_1) \geq \xi$. Since 
$q(\tau,\rho(\tau))=q^*(\tau)$, by (\ref{4.2}) and (\ref{4.4}) we 
have
\begin{align*}
    &B(\tau,\rho(\tau)) \\
     &\leq C_1\tau^{\kappa}
       + C_2 q^*(\tau) 
          + C_3 (q^*(\tau))^{2/3}+ C_4 (q^*(\tau))^{1/3} \\
     &\leq C_1\tau^{\kappa}
       + C_2 \Bigl( \frac{t_1}{\tau} \Bigr)^h q^*(t_1) 
          + C_3 \Bigl( \frac{t_1}{\tau} \Bigr)^{2h/3} q^*(t_1)^{2/3}
        + C_4 \Bigl( \frac{t_1}{\tau} \Bigr)^{h/3} q^*(t_1)^{1/3} \\
     &\leq C_1\tau^{\kappa}
       + C_2 \Bigl( \frac{t_1}{\tau} \Bigr)^h r 
          + C_3 \Bigl( \frac{t_1}{\tau} \Bigr)^{2h/3}r^{2/3}
        + C_4 \Bigl( \frac{t_1}{\tau} \Bigr)^{h/3} r^{1/3}
\end{align*}
for $t_1<\tau \leq t_0$, and so
\begin{align}
   \rho(t_1) \leq \xi + \int_{t_1}^{t_0}
       \Bigl(C_1\tau^{\kappa}
       &+ C_2 \Bigl( \frac{t_1}{\tau} \Bigr)^h r \label{4.6}\\
          &+ C_3 \Bigl( \frac{t_1}{\tau} \Bigr)^{2h/3}r^{2/3}
        + C_4 \Bigl( \frac{t_1}{\tau} \Bigr)^{h/3} r^{1/3}
           \Bigr) \frac{d\tau}{\tau}. \notag
\end{align}
Since
\[
     \int_{t_1}^{t_0}\Bigl( \frac{t_1}{\tau} \Bigr)^a
            \frac{d\tau}{\tau}
         = \frac{1}{a} \Bigl(1- \frac{{t_1}^a}{{t_0}^a} \Bigr)
         \leq \frac{1}{a}
\]
holds for any $a>0$, by applying this to (\ref{4.6}) we have 
(\ref{4.5}).
\end{proof}

\begin{cor}\label{Corollary2}
     If $\xi \in (0,R/2)$ we have $t_{\xi}=0$.
\end{cor}

\begin{proof}
   Let $\xi \in (0,R/2)$. Let us show that the condition $t_{\xi}>0$ 
leads us to a contradiction.
\par
   Suppose that $t_{\xi}>0$ holds. Then, by Lemmas 6 and 8 we have
\[
    0<\xi \leq \rho(t_1) \leq \xi+ \frac{C_1}{\kappa} \sigma^{\kappa}
         +\frac{C_2}{h}r + \frac{3C_3}{2h}r^{2/3} 
           + \frac{3C_4}{h}r^{1/3} =R_1< R
\]
for any $t_1$ satisfying $t_\xi<t_1<t_0$. Since 
$K=\{\rho \in (0,R) \,;\, \xi \leq \rho \leq R_1 \}$ is a compact 
subset of $(0,R)$ and since $\rho(t_1) \in K$ holds for any 
$t_1 \in (t_\xi,t_0]$, by a general theory of ordinary differential 
equations (for example, by Theorem 4.1 in \cite{CL}) we can extend 
$\rho(t)$ to $(t_{\xi}-\delta, t_0]$ for some $\delta>0$. This 
contradicts the assumption that $(t_{\xi}, t_0]$ is the interval of 
existence of a maximally extended solution $\rho(t)$. 
\end{proof}

   Step 4. Since $t_{\xi}=0$, by (\ref{4.4}) with $\tau=t_0$ 
we have
\[
    0 \leq q^*(t_0) \leq \Bigl( \frac{t_1}{t_0} \Bigr)^h q^*(t_1)
                  \leq \Bigl( \frac{t_1}{t_0} \Bigr)^h r
\]
for any $0<t_1<t_0$. Since $r>0$ is independent of $t_1$,
by letting $t_1 \longrightarrow 0$ we have $q^*(t_0)=0$.
Since $q^*(t_0)=q(t_0,\xi)$ we have $q(t_0,\rho)=0$ for any
$\rho \in (0,R/2)$. By the continuity we have $q(t_0,\rho)=0$ 
for any $\rho \in [0,R/2)$. 
Since $t_0 \in (0,\sigma]$ is taken arbitrarily 
we have $q(t,\rho)=0$ on $(0,\sigma] \times [0,R/2)$. 
\par
   This completes the proof of Proposition \ref{Proposition1}. 
\end{proof}

%
%
%  -----------------------------------------------------
%       << \S 5. Application >>
%  -----------------------------------------------------
%

\section{Application}\label{section5}
%\begin{center}
%{\bf \S 5. Application}
%\end{center}
%

   As an application of Theorem \ref{Theorem2}, we have

\begin{thm}[Analytic continuation]\label{Theorem3}
    Let $u(t,x)$ be a holomorphic solution of {\rm (\ref{2.1})} 
on $S_I(r_0) \times D_{R_0}$ for some nonempty open interval $I$, 
$r_0>0$ and $R_0>0$. If $u(t,x)$ satisfies {\rm (\ref{2.3})}, 
$u(t,x)$ can be continued holomorphically up to a neighborhood of 
$(0,0) \in \BC_t \times \BC_x^n$.
\end{thm}

%
%  --------------------------------
%        << References >>}
%  --------------------------------
%

%
%  --------------------------------
%        << End of Document >>
%  --------------------------------
%
\end{document}